\documentclass[12 pt]{amsart}
\usepackage{amscd,amssymb,amsthm}
\usepackage[arrow,matrix]{xy}
\usepackage{graphicx}
\theoremstyle{plain}
\newtheorem{thm}[subsection]{Theorem}

\newtheorem{cor}[subsection]{Corollary}

\theoremstyle{definition}
\numberwithin{equation}{section} 

\begin{document}
\title [Further Inequalities Between Vertex-Degree-Based Topological Indices ]{Further Inequalities Between Vertex-Degree-Based Topological Indices}
\author[A. Ali, A. A. Bhatti, Z. Raza]{Akbar Ali, Akhlaq Ahmad Bhatti and  Zahid Raza }
 %\address{Department of Mathematics \\ National University of Computer and Emerging Sciences, B-Block, Faisal Town, Lahore,         Pakistan.}
%\email{akbarali.maths@gmail.com, akhlaq.ahmad@nu.edu.pk, zahid.raza@nu.edu.pk}

%\subjclass{ ?????????????? }%

%\keywords{???????????????????????????????/}%

\maketitle

%\centerline{\scshape Akbar Ali$^1$, Akhlaq Ahmad Bhatti$^1$ and  Zahid Raza$^1$, Saleha Tariq$^1$ and M.Tariq Rahim$^2$}
{\footnotesize
\centerline{Department of Mathematics}
\centerline{National University of Computer and Emerging Sciences}
\centerline{Block-B, Faisal Town, Lahore, Pakistan}
\centerline{akbarali.maths@gmail.com, akhlaq.ahmad@nu.edu.pk, zahid.raza@nu.edu.pk} }
\medskip
\begin{abstract}
Continuing the recent work of L. Zhong and K. Xu [\textit{MATCH Commun. Math. Comput. Chem.} \textbf{71} (2014) 627-642], we determine inequalities among several vertex-degree-based
topological indices; first geometric-arithmetic index \textit{(GA)}, augmented Zagreb index \textit{(AZI)}, Randi$\acute{c}$ index \textit{(R)}, atom-bond connectivity index \textit{(ABC)}, sum-connectivity index \textit{(X)} and harmonic index \textit{(H)}.
\end{abstract}

\section{Introduction}
\begin{table}
\renewcommand{\arraystretch}{1.3}
\caption{Degree-based topological indices discussed in this paper}
\label{tab:example}
\centering
\begin{tabular}{|c|c|}
    \hline
    Name of index  &  Definition of index\\
    \hline

    Randi$\acute{c}$\textit{(R)}, \cite{4}-\cite{8} &   $R(G)=\displaystyle\sum_{ij\in E(G)}\frac{1}{\sqrt{d_{i}d_{j}}}$\\
    Harmonic\textit{(H)}, \cite{9}-\cite{12}  &   $H(G)=\displaystyle\sum_{ij\in E(G)}\frac{2}{d_{i}+d_{j}}$\\
    Atom-bond connectivity\textit{(ABC),} \cite{13}-\cite{19}  &   $ABC(G)=\displaystyle\sum_{ij\in E(G)}\sqrt{\frac{d_{i}+d_{j}-2}{d_{i}d_{j}}}$\\
    Sum-connectivity\textit{(X)}, \cite{20}-\cite{26}   &   $X(G)=\displaystyle\sum_{ij\in E(G)}\frac{1}{\sqrt{d_{i}+d_{j}}}$\\
    First geometric-arithmetic\textit{(GA)}, \cite{27,28}  &   $GA(G)=\displaystyle\sum_{ij\in E(G)}\frac{\sqrt{d_{i}d_{j}}}{\frac{1}{2}(d_{i}+d_{j})}$\\
    Augmented Zagreb \textit{(AZI)}, \cite{30}-\cite{33}   &   $AZI(G)=\displaystyle\sum_{ij\in E(G)}\left(\frac{d_{i}d_{j}}{d_{i}+d_{j}-2}\right)^{3}$\\
    \hline
\end{tabular}
\end{table}

Let $G=(V,E)$ denote a simple graph with vertex set $V(G)=\{v_{1},v_{2},...,v_{n}\}$ and edge set $E(G)$ such that $|E(G)|=m$. Suppose that $d_{i}$ is the degree of a vertex $v_{i}\in V(G)$ and $ij$ is edge connecting the vertices $v_{i}$ and $v_{j}$ \cite{1}.

Topological indices are numerical parameters of a graph which are invariant under graph isomorphisms. They play a significant role in mathematical chemistry especially in the QSPR/QSAR investigations \cite{2,3}.

L. Zhong and K. Xu \cite{29} obtained several inequalities among $R$, $ABC$, $X$ and $H$ indices. An important topological index that was not discussed in \cite{29} is the $AZI$ index. B. Furtula et al.\cite{30} proved that $AZI$ index is a valuable predictive index in the study of the heat of formation in octanes and heptanes. I. Gutman and J. To$\check{s}$ovi$\check{c}$ \cite{33} recently tested the correlation abilities of 20 vertex-degree-based topological indices for the case of standard heats of formation and normal boiling points of octane isomers, and they found that the augmented Zagreb index yield the best results. $GA$ index is another important topological index, not discussed in \cite{29}. It has been demonstrated, on the example of octane isomers, that $GA$ index is well-correlated with a variety of physico-chemical properties \cite{27}. For the mathematical properties of the $GA$ index and their applications in QSPR and QSAR see the survey\cite{28} and the references cited therein. In this note, we continue the work of L. Zhong and K. Xu \cite{29} and establish some inequalities among the topological indices given in Table 1.

\section{Inequalities Between Vertex-Degree-Based Topological Indices}

In this section, we give inequalities among several vertex-degree-based topological indices such as $AZI$, $GA$, $R$, $ABC$, $X$ and $H$ indices.

\begin{thm}\label{t1}
   If $G$ is a connected graph with $n\geq2$ vertices, then
\[\sqrt{2}X(G)\leq GA(G)\leq\sqrt{2(n-1)}X(G).\]
The lower bound is attained if and only if $G\cong P_{2}$ and the upper bound is attained if and only if $G\cong K_{n}$.
\end{thm}
 \begin{proof}
 Without loss of generality we can assume $1\leq d_{i}\leq d_{j}\leq n-1$. Consider the function
 \[F(x,y)=\left(\frac{\frac{2\sqrt{xy}}{x+y}}{\frac{1}{\sqrt{x+y}}}\right)^{2}=\frac{4xy}{x+y} \ \ \ \mbox{where} \ \ \ 1\leq x\leq y\leq n-1\]
 One can easily see that $F(x,y)$ is strictly monotone increasing in both $x$ and $y$. This implies that $F(x,y)$ attains the minimum value at $(x,y)=(1,1)$ and the maximum value at $(x,y)=(n-1, n-1)$. Hence
 \[2 = F(1,1)\leq F(x,y)\leq F(n-1,n-1)=2(n-1)\]
 which implies,
 \[\sqrt{2} \leq \frac{GA(G)}{X(G)}\leq \sqrt{2(n-1)}\]
 with the left equality if and only if $(d_{i},d_{j})=(1,1)$ for every edge $ij$ of $G$ and the
right equality if and only if $(d_{i},d_{j})=(n-1,n-1)$ for every edge $ij$ of $G$. This completes the proof.
\end{proof}
If graph $G$ has the minimum degree at least 2, then the lower bound in Theorem \ref{t1} can be improved:

\begin{cor}\label{c1}
If $G$ is a connected graph with minimum degree $\delta\geq2$, then
\[\sqrt{2\delta}X(G)\leq GA(G)\]
with equality if and only if $G$ is a $\delta$-regular graph.
\end{cor}

\begin{thm}\label{t2}
If $G$ is a connected graph with $n\geq2$ vertices, then
\[R(G)\leq GA(G)\leq(n-1)R(G).\]
The lower bound is attained if and only if $G\cong P_{2}$ and the upper bound is attained if and only if $G\cong K_{n}$.
\end{thm}

\begin{proof}
Suppose that $1\leq d_{i}\leq d_{j}\leq n-1$ and let

 \[F(x,y)=\frac{\frac{2\sqrt{xy}}{x+y}}{\frac{1}{\sqrt{xy}}}=\frac{2xy}{x+y} \ \ \ \mbox{where} \ \ \ 1\leq x\leq y\leq n-1. \]
 It can be easily seen that $F(x,y)$ is strictly monotone increasing in both $x$ and $y$. This implies that
 \[1 = F(1,1)\leq F(x,y)\leq F(n-1,n-1)=n-1\]
 therefore,
 \[1 \leq \frac{GA(G)}{R(G)}\leq (n-1)\]
 with the left equality if and only if $(d_{i},d_{j})=(1,1)$ for every edge $ij$ of $G$ and the
right equality if and only if $(d_{i},d_{j})=(n-1,n-1)$ for every edge $ij$ of $G$.
\end{proof}
If the graph $G$ has the minimum degree $\delta\geq2$, then the lower bound in Theorem \ref{t2} can be replaced by $\delta R(G)$.

\begin{cor}\label{c2}
If $G$ is a connected graph with $\delta\geq2$, then
\[\delta R(G)\leq GA(G)\]
with equality if and only if $G$ is a $\delta$-regular graph.
\end{cor}

B. Zhou and N. Trinajsti$\acute{c}$ \cite{23} proved that if $G$ is a connected graph with $n\geq3$ vertices, then $\sqrt{\frac{2}{3}}R(G)\leq X(G)$ with equality if and only if $G\cong
P_{3}$. Hence from Theorem \ref{t1}, we have:

\begin{cor}\label{c3}
If $G$ is a connected graph with with $n\geq3$ vertices, then
\[\sqrt{\frac{4}{3}} R(G)\leq GA(G)\leq(n-1)R(G)\]
with equality if and only if $G\cong P_{3}$ and right equality if and only if $G\cong K_{n}$.
\end{cor}

\begin{thm}\label{t3}
If $G$ is a connected graph with $n\geq2$ vertices, then
\[H(G)\leq GA(G)\leq(n-1)H(G).\]
The lower bound is attained if and only if $G\cong P_{2}$ and the upper bound is attained if and only if $G\cong K_{n}$.
\end{thm}

\begin{proof}
Using the same technique, used in proving Theorem \ref{t1} and Theorem \ref{t2}, one can easily prove the required result.
\end{proof}

\begin{cor}\label{c3}
     If $G$ is a connected graph with minimum degree $\delta\geq2$, then
\begin{equation}\label{1}
\delta H(G)\leq GA(G)
\end{equation}
with equality if and only if $G$ is a $\delta$-regular graph
\end{cor}

\begin{thm}\label{t4}
 If $G$ is a connected graph having $n\geq3$ vertices with minimum degree $\delta\geq2$, then
\begin{equation}\label{*}
\frac{\sqrt{2(n-2)}}{n-1}GA(G)\leq ABC(G)\leq\frac{n+1}{4\sqrt{n-1}}GA(G)
\end{equation}
with left equality if and only if $G\cong K_{n}$ and right equality if and only if $G\cong C_{3}$.
\end{thm}

\begin{proof}
Suppose that $2\leq d_{j}\leq d_{i}\leq n-1$ and consider the function
 \[F(x,y)=\left(\frac{\sqrt{\frac{x+y-2}{xy}}}{\frac{2\sqrt{xy}}{x+y}}\right)^{2}=\frac{(x+y)^{2}(x+y-2)}{4x^{2}y^{2}} \ \ \ \mbox{where} \ \ \ 2\leq y\leq x\leq n-1.\] Then
\[\frac{\partial F(x,y)}{\partial y}=-\frac{(x+y)\{x^{2}-y^{2}+x(x+y-4)\}}{4x^{2}y^{3}}\leq0\]
 This implies that $F(x,y)$ is monotone decreasing in $y$. Hence $F(x,y)$ attains the maximum value at $(x,y)=(x,2)$ for some $2\leq x\leq n-1$.\\
 But,
 \[\frac{dF(x,2)}{dx}=\frac{x(x^{2}-4)}{16x^{3}}\geq0\]
 that is, $F(x,2)$ is monotone increasing in $x$ which implies $F(x,2)$ has maximum value at $x=n-1$. Hence
 \[F(x,y)\leq F(n-1,2)=\frac{(n+1)^{2}}{16(n-1)}\]
 and therefore,
 \[\frac{ABC(G)}{GA(G)}\leq\frac{n+1}{4\sqrt{n-1}}\]
  with the equality if and only if $(d_{i},d_{j})=(n-1,2)$ for every edge $ij$ of $G$, i.e.,
  \[ABC(G)\leq\frac{n+1}{4\sqrt{n-1}}GA(G)\] with the equality if and only if $G\cong C_{3}$\\
  Since $F(x,y)$ is monotonously decreasing in $y$. It means $F(x,y)$ attains minimum value at $(x,x)$ for some $2\leq x\leq n-1$. Since
  \[\frac{dF(x,x)}{dx}=\frac{2x(2-x)}{x^{4}}\leq0\]
  hence,
  \[\frac{2(n-2)}{(n-1)^{2}}=F(n-1,n-1)\leq F(x,x)\leq F(x,y)\]
i.e.,
\[\frac{\sqrt{2(n-2)}}{(n-1)}\leq\frac{ABC(G)}{GA(G)}\]
   with the equality if and only if $(d_{i},d_{j})=(n-1,n-1)$ for every edge $ij$ of $G$, which completes the proof.
\end{proof}

L. Zhong and K. Xu \cite{29} proved that if $\delta\geq2$ in a connected graph $G$, then
 \begin{equation}\label{2}
H(G)\leq R(G)\leq X(G)< ABC(G)
\end{equation}
with the first equality if and only if G is a regular graph, and the second equality if and only if G is a cycle. Hence, from Theorem \ref{t4} and inequality (\ref{2}), we have:

\begin{cor}\label{c4}
     If $G$ is a connected graph with minimum degree $\delta\geq2$, then
\[ H(G)\leq R(G)\leq X(G)< ABC(G)\leq\frac{n+1}{4\sqrt{n-1}}GA(G)\] with the first equality if and only if G is a regular graph, the second equality if and only if G is a cycle, and last with equality if and only if $G\cong C_{3}$
\end{cor}

Denoted by $T^{*}$ the tree on eight vertices, obtained by joining the central vertices of two copies of star $K_{1,3}$ by an edge. K. C. Das and N. Trinajsti$\acute{c}$ \cite{18} proved that
\begin{equation}\label{3}
GA(G)>ABC(G)
\end{equation}
for every molecular graph $G\ncong K_{1,4}, T^{*}$. The same authors proved that inequality (\ref{3}) holds for any graph $G\ncong K_{1,4}, T^{*}$ in which $\Delta-\delta\leq3$. In \cite{19}, it is proved that if $\delta\geq2$ and $\Delta-\delta\leq(2\delta-1)^{2}$ then inequality (\ref{3}) holds.

\begin{cor}\label{c5}
If $G$ is a connected graph satisfying at least one of the following properties:
\begin{description}
  \item[{(i)}] $G$ is molecular graph such that $G\ncong K_{1,4}, T^{*}$
  \item[{(ii)}] $\Delta-\delta\leq3$ and $G\ncong K_{1,4}, T^{*}$
  \item[{(iii)}] $\delta\geq2$ and $\Delta-\delta\leq(2\delta-1)^{2},$
\end{description}
then
\[H(G)\leq R(G)\leq X(G)< ABC(G)<GA(G)\]
\end{cor}
Denote the chromatic number of a graph $G$ by $\chi(G)$. Deng et al. \cite{11} proved that for every connected graph $G$
\begin{equation}\label{4}
\chi(G)\leq2H(G)
\end{equation}
with equality if and only if $G$ is a complete graph. From inequalities (\ref{1}) and (\ref{4}), we obtain a sharp upper bound of $\chi(G)$ in terms of GA index:

\begin{cor}\label{c6}
If $G$ is a connected graph of order $n$ with minimum degree $\delta\geq2$, then
\[ \chi(G)\leq\frac{2}{\delta} GA(G)\]
with equality if and only if $G\cong K_{n}$.
\end{cor}

Another vertex-degree-based topological Index is the modified second Zagreb index defined \cite{34,35} as:\\
\[M_{2}^{*}(G)=\sum_{ij\in E(G)}\frac{1}{d_{i}d_{j}} \]
Using the same technique, used in proving Theorem \ref{t1} and Theorem \ref{t2}, one can easily prove the following result:

\begin{thm}\label{t5}
If $G$ is a connected graph with $n\geq2$ vertices, then
\begin{equation}\label{5}
M_{2}^{*}(G)\leq R(G)\leq(n-1)M_{2}^{*}(G),
\end{equation}
\begin{equation}\label{6}
\frac{M_{2}^{*}(G)}{\sqrt{2}}\leq X(G)\leq\frac{(n-1)^{\frac{3}{2}}}{\sqrt{2}}M_{2}^{*}(G),
\end{equation}
\begin{equation}\label{7}
M_{2}^{*}(G)\leq H(G)\leq(n-1)M_{2}^{*}(G),
\end{equation}
\begin{equation}\label{8}
M_{2}^{*}(G)\leq GA(G)\leq(n-1)^{2}M_{2}^{*}(G),
\end{equation}
\begin{equation}\label{9}
\sqrt{2}M_{2}^{*}(G)\leq ABC(G)\leq(n-1)\sqrt{2(n-2)}M_{2}^{*}(G); n\geq3.
\end{equation}
The left equality in (\ref{5})-(\ref{8}) and in (\ref{9}) is attained if and only if $G\cong P_{2}$ and $G\cong P_{3}$ respectively. The right equality in all inequalities (\ref{5})-(\ref{9}) is attained if and only if $G\cong K_{n}$.
\end{thm}

\begin{cor}\label{c7}
If $G$ is a connected graph with minimum degree $\delta\geq2$, then
\begin{equation}\label{10}
\delta M_{2}^{*}(G)\leq R(G),
\end{equation}
\begin{equation}\label{11}
\frac{\delta^{\frac{3}{2}}M_{2}^{*}(G)}{\sqrt{2}}\leq X(G),
\end{equation}
\begin{equation}\label{12}
\sqrt{\delta}M_{2}^{*}(G)\leq H(G),
\end{equation}
\begin{equation}\label{13}
\delta^{2}M_{2}^{*}(G)\leq GA(G),
\end{equation}
\begin{equation}\label{14}
\delta\sqrt{2(\delta-1)}M_{2}^{*}(G)\leq ABC(G).
\end{equation}
The equality in all inequalities (\ref{10})-(\ref{14}) is attained if and only if $G$ is a $\delta$-regular graph.
\end{cor}
Now, we establish some inequalities between augmented Zagreb index and other vertex-degree-based
topological indices.

\begin{thm}\label{t6}
 If $G$ is a connected graph having $n\geq3$ vertices, then
\[\frac{1536}{343}X(G)\leq AZI(G)\leq\frac{\sqrt{(n-1)^{13}}}{\sqrt{32}(n-2)^{3}}X(G)\]
with left equality if and only if $G\cong S_{1,8}$ and right equality if and only if $G\cong K_{n}$.
\end{thm}
\begin{proof}
 Without loss of generality we can assume $1\leq d_{i}\leq d_{j}\leq n-1$. Consider the function
\[ F(x,y)=\left(\frac{\left(\frac{xy}{x+y-2}\right)^{3}}{\frac{1}{\sqrt{x+y}}}\right)^{2}=(x+y)\left(\frac{xy}{x+y-2}\right)^{6} \ \  \mbox{with}  \ \ 1\leq x\leq y\leq n-1 \ \  \mbox{and} \ \ y\geq2.\] Then,
\[\frac{\partial F(x,y)}{\partial x}=\frac{x^{5}y^{6}\{x^{2}+(y-2)(6y+7x)\}}{(x+y-2)^{7}}\geq0.\]
This means $F(x,y)$ is increasing in $x$ and hence is minimum at $(1,y_{1})$ and maximum at $(y_{2},y_{2})$ for some $2\leq y_{1},y_{2}\leq n-1$. Now,
 \[\frac{dF(1,y)}{dy}=\frac{y^{5}(y^{2}-7y-6)}{(y-1)^{7}}\]
 and this implies, $F(1,y)$ is monotone decreasing in $2\leq y\leq7$ and monotone increasing in $8\leq y\leq n-1$. Hence minimum value of $F(x,y)$ is
 \[\min\{F(1,7), F(1,8)\}=9\left(\frac{8}{7}\right)^{6}\]
 Therefore,
\begin{equation}\label{15}
3\left(\frac{8}{7}\right)^{3}\leq \frac{AZI(G)}{X(G)}
\end{equation}
 with equality if and only if $(d_{i},d_{j})=(1,8)$ for each edge $ij$ of $G$.\\
 Moreover,
  \[F(y,y)=\frac{y^{13}}{32(y-1)^{6}}\]
  is monotone increasing and hence
  \[F(x,y)\leq F(n-1,n-1)=\frac{(n-1)^{13}}{32(n-2)^{6}}\]
  therefore,
  \begin{equation}\label{16}
\frac{AZI(G)}{X(G)}\leq\frac{\sqrt{(n-1)^{13}}}{\sqrt{32}(n-2)^{3}}
\end{equation}
with equality if and only if $(d_{i},d_{j})=(n-1,n-1)$ for each edge $ij$ of $G$. From (\ref{15}) and (\ref{16}), required result follows.
\end{proof}

If graph $G$ has the minimum degree at least 2, then the lower bound in Theorem \ref{t6} can be improved:

\begin{cor}\label{c8}
If $G$ is a connected graph with minimum degree $\delta\geq2$, then
\[\frac{\delta^{\frac{13}{2}}}{\sqrt{32}(\delta-1)^{3}}X(G)\leq AZI(G)\]
with equality if and only if $G$ is a $\delta$-regular graph.
\end{cor}
Using the similar technique, used in proving Theorem \ref{t6}, one can prove the following result (we omit the proof)
\begin{thm}\label{t7}
If $G$ is a connected graph with $n\geq3$ vertices, then
\begin{equation}\label{17}
\frac{343\sqrt{7}}{216}R(G)\leq AZI(G)\leq\frac{(n-1)^{7}}{8(n-2)^{3}}R(G),
\end{equation}
\begin{equation}\label{18}
\frac{375}{64}H(G)\leq AZI(G)\leq\frac{(n-1)^{7}}{8(n-2)^{3}}H(G),
\end{equation}
\begin{equation}\label{19}
\left(\frac{n-1}{n-2}\right)^{\frac{7}{2}}ABC(G)\leq AZI(G)\leq \left(\frac{(n-1)^{2}}{2(n-2)}\right)^{\frac{7}{2}}ABC(G),
\end{equation}
\begin{equation}\label{20}
8GA(G)\leq AZI(G)\leq \frac{(n-1)^{6}}{8(n-2)^{3}}GA(G) ; \ \ \delta\geq2,
\end{equation}
\begin{equation}\label{21}
4M_{2}^{*}(G)\leq AZI(G)\leq\frac{(n-1)^{4}}{2(n-2)}M_{2}^{*}(G).
\end{equation}
The left equality in (\ref{17}), (\ref{18}), (\ref{19}), (\ref{20}), (\ref{21}) hold if and only if $G\cong S_{1,7}$, $G\cong S_{1,5}$, $G\cong S_{1,n-1}$, $G\cong C_{n}$, $G\cong P_{3}$ respectively and right equality if and only if $G\cong K_{n}$.
\end{thm}

\begin{cor}\label{c9}
If $G$ is a connected graph with $\delta\geq2$, then
\begin{equation}\label{22}
\frac{\delta^{7}}{8(\delta-1)^{3}}R(G)\leq AZI(G),
\end{equation}
\begin{equation}\label{23}
\frac{\delta^{7}}{8(\delta-1)^{3}}H(G)\leq AZI(G),
\end{equation}
\begin{equation}\label{24}
\left(\frac{\delta^{2}}{2(\delta-1)}\right)^{\frac{7}{2}}ABC(G)\leq AZI(G),
\end{equation}
\begin{equation}\label{25}
\frac{\delta^{6}}{8(\delta-1)^{3}}GA(G)\leq AZI(G),
\end{equation}
\begin{equation}\label{26}
\frac{\delta^{4}}{2(\delta-1)}M_{2}^{*}(G)\leq AZI(G).
\end{equation}
The equality in all inequalities (\ref{22})-(\ref{26}) hold if and only if $G$ is a $\delta$-regular graph.
\end{cor}

\end{document}